\documentclass[10pt]{amsart}
\usepackage{amsmath,amsfonts,amssymb}
\usepackage[colorlinks=true, linkcolor=blue,
  citecolor=blue, urlcolor=blue]{hyperref}
\usepackage{tableau}
\usepackage{subcaption}
\usepackage{mathtools}
\usepackage{float}
\usepackage[alphabetic]{amsrefs}
\usepackage{tikz}
\usepackage{enumitem}
\usepackage[margin=1in]{geometry}

\def\newthm#1#2{\newtheorem{#1}[dummy]{#2}%
  \expandafter\def\csname#2\endcsname##1{\hyperref[#1:##1]{#2~\ref*{#1:##1}}}}
\newthm{thm}{Theorem}
\newthm{lemma}{Lemma}
\newthm{prop}{Proposition}
\newthm{cor}{Corollary}
\newthm{conj}{Conjecture}
\newthm{cons}{Consequence}
\newthm{fact}{Fact}
\newthm{obs}{Observation}
\newthm{guess}{Guess}
\theoremstyle{definition}
\newthm{defn}{Definition}
\newthm{remark}{Remark}
\newthm{example}{Example}
\newthm{question}{Question}
\newthm{notation}{Notation}

\makeatletter
\def\namedlabel#1#2{\begingroup
    #2%
    \def\@currentlabel{#2}%
    \phantomsection\label{#1}\endgroup
}
\makeatother

\newcommand{\Section}[1]{\hyperref[sec:#1]{Section~\ref*{sec:#1}}}
\newcommand{\Table}[1]{\hyperref[tab:#1]{Table~\ref*{tab:#1}}}
\newcommand{\eqn}[1]{\hyperref[eqn:#1]{(\ref*{eqn:#1})}}
\newcommand{\Figure}[1]{\hyperref[fig:#1]{Figure~\ref*{fig:#1}}}

\DeclareMathOperator{\Gr}{Gr}

\DeclareMathOperator{\Span}{Span}

\DeclareMathOperator{\Sing}{Sing}

\DeclareMathOperator{\codim}{codim}

\DeclareMathOperator{\ch}{\ch}

\usepackage{bm}

\newcommand{\C}{{\mathbb C}}

\newcommand{\id}{\text{id}}

\newcommand{\ignore}[1]{}

\begin{document}

\title
[The Isomorphism Problem for Grassmannian Schubert Varieties]
{The Isomorphism Problem for Grassmannian Schubert Varieties}

\date{\today}

\author{Mihail Țarigradschi}
\address{Department of Mathematics, Rutgers University, 110
    Frelinghuysen Road, Piscataway, NJ 08854, USA}
\email{mt994@math.rutgers.edu}

\author{Weihong Xu}
\address{Department of Mathematics (0123), Virginia Tech, 460 McBryde Hall,
    225 Stanger Street,
    Blacksburg, VA 24061-1026}
\email{weihong@vt.edu}

% \subjclass[2010]{Primary 14N35; Secondary 14M15, 05E05, 14N15}

% \thanks{The authors were supported in part by NSF grant DMS-1503662.}

\begin{abstract}
    We prove that Schubert varieties in potentially different Grassmannians are isomorphic as varieties if and only if their corresponding Young diagrams are identical up to a transposition. We also discuss a generalization of this result to Grassmannian Richardson varieties. In particular, we prove that Richardson varieties in potentially different Grassmannians are isomorphic as varieties if their corresponding skew diagrams are semi-isomorphic as posets, and we conjecture the converse. Here, two posets are said to be semi-isomorphic if there is a bijection between their sets of connected components such that the corresponding components are either isomorphic or opposite.
\end{abstract}

\maketitle

\section{Introduction}

The Grassmannian \(\Gr(m,n)\) of \(m\) dimensional subspaces of \(\C^n\) and Schubert varieties (natural subvarieties of the Grassmannian) are central and classical objects of study in geometry, with connections to representation theory and algebraic combinatorics. Schubert varieties of \(\Gr(m,n)\) are indexed by Young diagrams contained in a rectangle with \(m\) rows and \(n-m\) columns of boxes.

In this note, we prove the following.

\begin{thm}[\(\doteq \Theorem{main}\)]\label{thm:schub}
    Schubert varieties in potentially different Grassmannians are isomorphic as varieties if and only if their corresponding Young diagrams are identical up to a transposition.
\end{thm}

The isomorphism problem for Schubert varieties in complete flag varieties of Kac-Moody type has been studied by Richmond and Slofstra \cite{richmond2021isomorphism}. Develin, Martin, and Reiner classified a class of smooth Schubert varieties in type \(A\) partial flag varieties \cite{develin_martin_reiner_2007}.

This problem can be extended to Richardson varieties. A Richardson variety is the intersection of a Schubert variety and an opposite Schubert variety in a flag variety. In the Grassmannian case, the two Schubert varieties correspond to two Young diagrams. When the first diagram contains the latter, the difference of these diagrams is called a skew diagram. Non-empty Grassmannian Richardson varieties are indexed by skew diagrams.

To a skew diagram we assign a poset structure on the set of its boxes, see \Section{rich}.

For Grassmannian Richardson varieties, we prove the following statement and conjecture its converse.

\begin{prop}[\(\doteq \Proposition{if}\)]\label{prop:rich}
    If two skew diagrams are semi-isomorphic as posets, then the corresponding Richardson varieties in potentially different Grassmannians are isomorphic as varieties.
\end{prop}

Here, two posets are said to be semi-isomorphic if there is a bijection between their sets of connected components such that the corresponding components are either isomorphic or opposite. When specialized to the case of Schubert varieties, \Proposition{rich} says that if two Young diagrams are identical up to a transpose, then the corresponding Schubert varieties in potentially different Grassmannians are isomorphic as varieties.
This is the easier direction of \Theorem{schub}.

To prove the other direction of \Theorem{schub}, we use the combinatorial description of the singular locus of a Schubert variety given by Lakshmibai and Weyman \cite{singloci}, and the well-known fact that classes of the Schubert subvarieties contained in a Schubert variety form a basis for the Chow ring of that Schubert variety. These tools allow us to attack the problem combinatorially.

In \Section{prelim} we set up notations and state the preliminaries; we prove \Theorem{schub} in \Section{schub}; \Section{rich} is devoted to discussing the isomorphism problem for Grassmannian Richardson varieties.

\subsection*{Acknowledgements} We thank Anders Buch for helpful comments on our exposition.

\section{Notation and Preliminaries}\label{sec:prelim}

Let \(X = \Gr(m,n)\) be the Grassmannian of \(m\) dimensional subspaces of \(\C^n\). Schubert varieties in \(X\) are indexed by partitions \( \lambda = (\lambda_1, \dots, \lambda_m) \) such that \(k := n-m \geq \lambda_1 \geq \dots \geq \lambda_m \geq 0\) is a decreasing sequence of integers. We will identify such a partition with the corresponding Young diagram, whose \(i\)-th row from the top contains \(\lambda_i\) boxes.
We call a Young diagram \(\lambda' = (\lambda'_1, \dots, \lambda'_m)\) a subdiagram of \(\lambda\) if \( \lambda'_i \leq \lambda_i\) for \(i=1, \dots, m\) (see \Figure{subdiagram} for an example). Given two Young diagrams \(\lambda, \mu\) we define their intersection \(\lambda \cap \mu\) as the Young diagram consisting of boxes in both \(\lambda\) and \(\mu\).
For brevity, we will refer to Young diagrams simply as diagrams.

\begin{figure}[H]
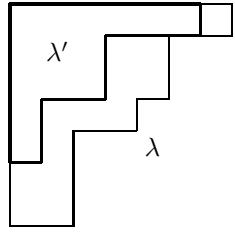

    \[ \tableau{12}{[TL]& [T]& [T]& [T]& [H]& [HR]& [hr]\\ [L]& []\lambda'& []& [LT]& [r]\\[L]& []& [R]& []& [br]\\ [LR]& [T]& [Tb]& [br]\\ [LR]& [r] &[]&[]&[]\lambda \\ [Tl]& [r]\\ [bl]& [br]}\]
    \caption{The diagram \(\lambda'\) is a subdiagram of \(\lambda\).}
    \label{fig:subdiagram}
\end{figure}
We will also write \(\lambda\) as \(({a_1}^{b_1}, \dots, {a_r}^{b_r})\), where \( a_1\geq \dots \geq a_r \geq 0 \) and \( b_i \geq 0 \) for \(i=1,\dots,r\). Here \(a^b\) means \(a\) repeated \(b\) times, corresponding to a rectangle with \(b\) rows and \(a\) columns of boxes. Whenever \(b=0\) or \(a=0\), the term can be left out of the expression.
Note that such an expression for \(\lambda\) is not unique. However, if we impose the extra conditions that \(a_1 > \dots > a_r > 0\) and \(b_i>0\) for all \(i = 1, \dots, r\), then the expression is unique, in which case we say that \(\lambda\) consists of \(r\) rectangles.

Let \(e_1, \dots , e_n\) be the standard basis of \(\C^n\). Given a partition \( \lambda = (\lambda_1, \dots, \lambda_m) \) that satisfies \( k \geq \lambda_1 \geq \dots \geq \lambda_m \geq 0 \), the corresponding Schubert variety is
\[
    X_\lambda = \{ \Sigma\in X: \dim(\Sigma \cap \Span\{ e_1, \dots, e_{\lambda_i + m-i+1} \}) \geq m-i+1\ \text{for } i=1, \dots, m \},
\]
which has complex dimension \[ \dim(X_\lambda) = | \lambda | \coloneqq \lambda_1+\dots+\lambda_m;\] and the corresponding opposite Schubert variety is
\[
    X^\lambda = \{ \Sigma\in X: \dim(\Sigma \cap \Span\{ e_n, e_{n-1}, \dots, e_{\lambda_i + m-i+1} \}) \geq i\ \text{for } i=1, \dots, m \},
\]
which has complex codimension \( \codim(X^\lambda) = | \lambda | \) in \(X\).

We shall repeatedly use the following well-known lemma.

\begin{lemma}\label{lemma:transpose}
    There is an isomorphism between \(X=\Gr(m,n)\) and \(Y=\Gr(n-m,n)\), inducing isomorphisms between \(X_\lambda\) and \(Y_{\lambda^T}\) and between \(X^\lambda\) and \(Y^{\lambda^T}\).
\end{lemma}
\begin{proof}
    Note that \(\Sigma\) is an \(m\)-dimensional subspace of \(\C^n\) if and only if \((\C^n/\Sigma)^*\) is an \((n-m)\)-dimensional subspace of \((\C^n)^*\). The isomorphism is given by sending \(\Sigma\) to \((\C^n/\Sigma)^*\).
\end{proof}

We let \[\lambda^\vee\coloneqq(k-\lambda_m,\dots,k-\lambda_1)\] denote the Poincar\'e dual partition of \(\lambda\).

The automorphism \( w_0 : \C^n \to \C^n \) given by \( w_0(e_i) = e_{n+1-i} \) induces an automorphism \(w_0\) of \(X\), furthermore it can be seen that \( w_0(X^\lambda) = X_{\lambda^\vee} \).

The irreducible components of the singular locus of \(X_\lambda\), \(\Sing(X_\lambda)\), are then characterized by \cite[Theorem 5.3]{singloci}, which we restate as \Theorem{sing} and illustrate in \Figure{sing}. In the following, let \(\lambda = ({a_1}^{b_1}, \dots, {a_r}^{b_r})\), where \( a_1 > \dots > a_r > 0 \) and \( b_i > 0 \) for \(i=1,\dots,r\).

\begin{thm}\label{thm:sing}
    \(\Sing(X_\lambda)\) has \(r-1\) irreducible components \(X_{\lambda^1},\dots,X_{\lambda^{r-1}}\), where \[\lambda^i=({a_1}^{b_1},\dots,{a_{i-1}}^{b_{i-1}},{a_i}^{b_i-1},(a_{i+1}-1)^{a_{i+1}+1},{a_{i+2}}^{b_{i+2}},\dots,{a_r}^{b_r})\text{ for }i=1,\dots,r-1.\]
\end{thm}

\begin{figure}[h]
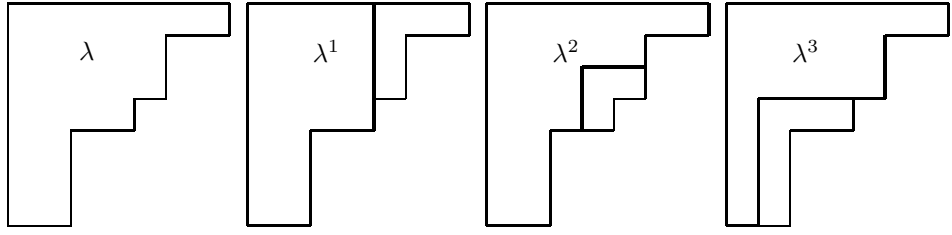

    \[\tableau{12}{ [tl]& [t]& [t]& [t]& [t]& [h]& [hr]\\ [l]& []& []\lambda& []& [r]\\[l]& []& [] & []& [br]\\ [l]& []& [b]& [br]\\ [l]& [r]\\ [l]& [r]\\ [bl]& [br]}
        \ \ \tableau{12}{ [TL]& [T]& [T]& [TR]& [t]& [h]& [hr]\\ [L]& []& []\lambda^1& [R]& [r]\\[L]& & []& [R]& [br]\\ [L]& []& [B]& [BR]\\ [L]& [R]\\ [L]& [R]\\ [BL]& [BR]}
        \ \ \tableau{12}{ [TL]& [T]& [T]& [T]& [T]& [H]& [HR]\\ [L]& []& []\lambda^2& []& [R]\\[L]& []& [R]& [T]& [Tbr]\\ [L]& []& [BR]& [br]\\ [L]& [R]\\ [L]& [R]\\ [BL]& [BR]}
        \ \ \tableau{12}{ [TL]& [T]& [T]& [T]& [T]& [H]& [HR]\\ [L]& []& []\lambda^3& []& [R]\\[L]& []& []& []& [BR]\\ [LR]& [T]& [Tb]& [Tbr]\\ [LR]& [r]\\ [LR]& [r]\\ [BV]& [br]}\]
    \caption{Diagrams \(\lambda^1,\dots,\lambda^{r-1}.\)}
    \label{fig:sing}
\end{figure}

Note that each \(\lambda^i\) is obtained by removing a hook on the boundary of \( \lambda \).
Furthermore, \((\lambda^T)^i=(\lambda^{r-i})^T\) for \(i=1, \dots, r-1\), where \(\lambda^T\) denotes the conjugate (or transpose) partition of \(\lambda\).
Using the fact that \( X_\lambda \cap X_\mu = X_{\lambda \cap \mu} \), we have that \( X_{\lambda^i}\) and \(X_{\lambda^j} \) intersect properly as subvarieties of \(X_\lambda\) if and only if the removed hooks are disjoint.

\begin{notation}
    We will use \(\lambda^0\) to denote the partition such that \(X_{\lambda^0}\) is the intersection of all irreducible components of \(\Sing(X_\lambda)\).
\end{notation}
It follows from \Theorem{sing} that \(\lambda^0 = \lambda^1 \cap \lambda^2 \cap \dots \cap \lambda^{r-1} \).
Note that \( (\lambda^0)^T = (\lambda^T)^0 \).
See \Figure{lambda0} for an illustration of \(\lambda^0\). If \( \lambda = ({a_1}^{b_1}, \dots, {a_r}^{b_r}) \) as above, then
\[
    \lambda^0 = ({a_1}^{b_1-1},(a_2-1)^{b_2},\dots,(a_{r-1}-1)^{b_{r-1}},(a_r-1)^{b_r+1}).
\]

\begin{figure}[H]
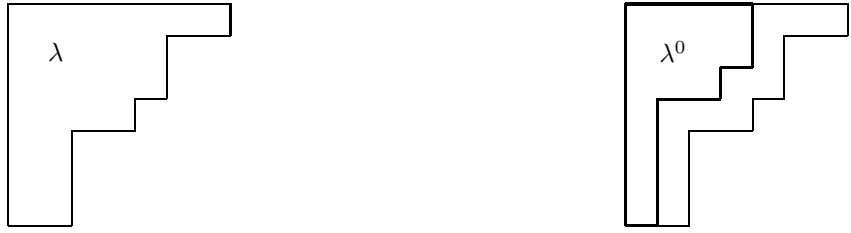

    \begin{subfigure}{0.49\textwidth}
        \[\tableau{12}{ [tl]& [t]& [t]& [t]& [t]& [h]& [hr]\\ [l]& []\lambda& []& []& [r]\\[l]& []& [] & []& [br]\\ [l]& []& [b]& [br]\\ [l]& [r]\\ [l]& [r]\\ [bl]& [br]}\]
    \end{subfigure}
    \begin{subfigure}{0.49\textwidth}
        \[\tableau{12}{ [TL]& [T]& [T]& [TR]& [t]& [h]& [hr]\\ [L]& []\lambda^0& []& [BR]& [r]\\[L]& []& [R]& []& [br]\\ [LR]& [T]& [Tb]& [br]\\ [LR]& [r]\\ [LR]& [r]\\ [BV]& [br]}
        \]
    \end{subfigure}
    \caption{The diagram \(\lambda^0\).}
    \label{fig:lambda0}
\end{figure}

\begin{defn}
    Let \( \xi(\lambda) \) be the size of the longest hook contained in \( \lambda \).
\end{defn}

Note that \( \xi(\lambda) = \xi(\lambda^T) \).
If we represent \( \lambda\) as \(({a_1}^{b_1}, \dots, {a_r}^{b_r}) \) with \( a_r, b_1 > 0 \), then \[ \xi(\lambda) = a_1 + b_1 + \dots + b_r - 1 .\]

\section{The isomorphism problem for Schubert varieties}\label{sec:schub}

One direction of \Theorem{main} is straightforward. We prove it in \Lemma{forward}.

\begin{lemma}\label{lemma:forward}
    Let \(X,\ X'\) be two Grassmannians and let \(X_\lambda\subseteq X,\ X'_\mu\subseteq X'\) be Schubert varieties. If \(\lambda=\mu\text{ or }\lambda^T = \mu\), then \(X_\lambda\cong X'_\mu\) as varieties.
\end{lemma}
\begin{proof}
    By \Lemma{transpose}, we may assume without loss of generality that \( \lambda = \mu \). Let \(m,n,m',n'\) be such that \[X \cong \Gr(m,n)\text{ and }X' \cong \Gr(m', n').\] Let \(M,N\) be such that \[M \geq \max(m,m')\text{ and }N-M \geq \max(n-m,n'-m')\] and let \[Z = \Gr(M,N).\] We identify \(X \) with the set of subspaces \(\Sigma\) of \( \Span\{ e_{M-m+1}, \dots, e_{M-m+n} \}\) of dimension \(m\).
    The closed embedding \(X \hookrightarrow Z\) given by
    \[
        \Sigma \mapsto \Span\{ e_1, \dots, e_{M-m} \} \oplus \Sigma.
    \]
    induces an isomorphism of Schubert varieties \(X_\lambda \cong Z_\lambda\). Similarly, \( X'_\mu \cong Z_\mu = Z_\lambda \).
\end{proof}

We now work on the converse. Let \(X,\ X'\) be two Grassmannians and let \(X_\lambda\subseteq X,\ X'_\mu\subseteq X'\) be Schubert varieties. Assume \(X_\lambda\cong X'_\mu\) as varieties. Then \(\Sing(X_\lambda)\cong\Sing(X'_\mu)\). By \Theorem{sing}, this implies that \(\lambda\) and \(\mu\) consist of the same number of rectangles. We denote this number by \(r\). Moreover, we must have \(X_{\lambda^0}\cong X'_{\mu^0}\) as varieties. By the same reasoning, \(\lambda^0\) and \(\mu^0\) consist of the same number of rectangles. We denote this number by \(r_0\). In addition, we have an isomorphism of Chow groups \(A_i(X_\lambda)\cong A_i(X'_\mu)\) for all \(i\). Size \(i\) sub-diagrams of \(\lambda\) correspond to dimension \(i\) Schubert varieties contained in \(X_\lambda\), whose classes form a basis for \(A_i(X_\lambda)\), and an analogous statement holds for sub-diagrams of \(\mu\). Therefore, \(\lambda\) and \(\mu\) must have the same number of size \(i\) sub-diagrams for all \(i\). In particular, \(|\lambda|= \dim X_\lambda = \dim X'_\mu = |\mu| \) as the size of the largest sub-diagram.

\begin{prop}\label{prop:converse}
    Let \(X,\ X'\) be two Grassmannians and let \(X_\lambda\subseteq X,\ X'_\mu\subseteq X'\) be Schubert varieties. If \(X_\lambda\cong X'_\mu\) as varieties, then \(\lambda=\mu\text{ or }\lambda^T = \mu\).
\end{prop}
\begin{proof}
    Note that \(r-2\leq r_0\leq r\). The proposition follows from \Lemma{nonsing}, \Lemma{0}, \Lemma{1}, and \Lemma{2}.
\end{proof}

\begin{lemma}\label{lemma:nonsing}
    \Proposition{converse} holds when \(X_\lambda\) and \(X'_\mu\) are nonsingular.
\end{lemma}
\begin{proof}
    Note that being nonsingular is equivalent to \(r \leq 1\).
    When \(r=0\), then \(\lambda=\mu=\varnothing\).

    When \(r = 1\), we have \[\lambda=(a^b),\ \mu=({a'}^{b'})\] for some positive integers \(a, b, a', b'\) such that \[ab=|\lambda|=|\mu|=a'b'.\] Up to transposing some diagrams, we may assume \(a\leq b\) and \(a'\leq b'\). Suppose \( a > a' \), then every diagram of size \(a'+1\) is a sub-diagram of \(\lambda\), but the single-row diagram \((a'+1)\) is not a sub-diagram of \(\mu\) (see \Figure{rect}), which is a contradiction. Therefore, \(a=a'\) and \(b=b'\).

    \begin{figure}[h]
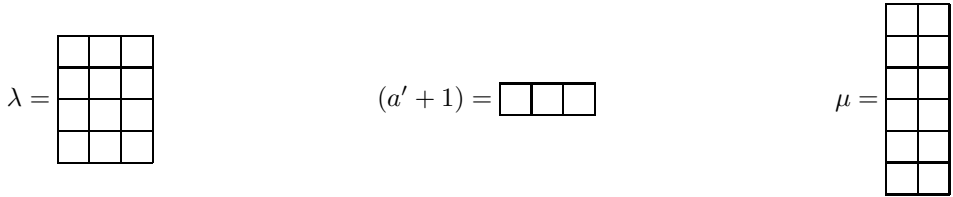

        \begin{subfigure}{0.32\textwidth}
            \centering
            \(\lambda=\tableau{12}{ [a]& [a]& [a]\\ [a]& [a]& [a]\\[a]& [a]& [a]\\[a]& [a]& [a]}\)
        \end{subfigure}
        \begin{subfigure}{0.32\textwidth}
            \centering
            \((a'+1)=\tableau{12}{ [a]& [a]& [a]}\)
        \end{subfigure}
        \begin{subfigure}{0.32\textwidth}
            \centering
            \(\mu=\tableau{12}{ [a]& [a]\\ [a]& [a]\\[a] & [a]\\[a]& [a]\\[a]& [a]\\[a]& [a]\\}\)
        \end{subfigure}
        \caption{Diagrams \(\lambda,\  (a'+1),\ \mu\).}
        \label{fig:rect}
    \end{figure}
\end{proof}

We will prove \Lemma{0}, \Lemma{1}, and \Lemma{2} by induction on \(|\lambda|\).

\begin{lemma}\label{lemma:0}
    \Proposition{converse} holds when \(r_0=r\geq 2\).
\end{lemma}

\begin{proof}
    In this case,
    \[
        \lambda=({a_1}^{b_1},\dots,{a_r}^{b_r}),\ \mu=({a_1'}^{b_1'},\dots,{a_r'}^{b_r'}),
    \]
    with \( a_1 > a_2 > \dots >a_r>1\); \(a_1' > a_2' > \dots > a_r'>1\); \(b_1,\ b_1'>1\); and \( b_i,\ b_i' > 0 \) for \(i=2,\dots,r\). Note that
    \[
        \lambda^0=({a_1}^{b_1-1},(a_2-1)^{b_2},\dots,(a_{r-1}-1)^{b_{r-1}},(a_r-1)^{b_r+1})
    \]
    (this expression contains no zeroes), and an analogous expression holds for \(\mu^0\). Since \(X_{\lambda^0}\cong X'_{\mu^0}\) and they are of smaller dimension, by induction hypothesis, we have
    \[
        \text{either } \lambda^0=\mu^0\text{ or }(\lambda^T)^0=(\lambda^0)^T=\mu^0.
    \]
    Up to transposing \(\lambda\), we may assume \(\lambda^0=\mu^0\), which implies \(a_i=a_i',\ b_i=b_i'\text{ for all }i\); in other words, \(\lambda=\mu\).
\end{proof}

\begin{lemma}\label{lemma:1}
    \Proposition{converse} holds when \(r_0+1=r\geq 2\).
\end{lemma}

\begin{proof}
    Up to transposing \(\lambda\) and \(\mu\), we may assume
    \[
        \lambda=(a_1,{a_2}^{b_2},\dots,{a_r}^{b_r}),\ \mu=(a_1',{a_2'}^{b_2'},\dots,{a_r'}^{b_r'}),
    \]
    with \( a_1 > a_2 > \dots >a_r>1\); \(a_1' > a_2' > \dots > a_r'>1\);  and \( b_i,\ b_i' > 0 \) for \(i=2,\dots,r\).

    Suppose \(r=2\).

    We have \[\lambda=(a_1, {a_2}^{b_2}),\ \mu=(a_1',{a_2'}^{b_2'}),\ \lambda^0=((a_2-1)^{b_2+1}),\text{ and }\mu^0=((a_2'-1)^{b_2'+1}).\] By induction hypothesis, we have \[\text{either }\lambda^0=\mu^0\text{ or }(\lambda^T)^0=(\lambda^0)^T=\mu^0.\]

    If \(\lambda^0=\mu^0\), then \(a_2=a_2'\) and \(b_2=b_2'\). Since \(|\lambda|=|\mu|\), we must also have \(a_1=a_1'\) and therefore \(\lambda=\mu\).

    If \((\lambda^0)^T=\mu^0\), then we have
    \begin{equation}\label{eqn:two=s}
        a_2-1=b_2'+1\text{ and }a_2'-1=b_2+1.
    \end{equation}
    In this case, if \(a_2\leq b_2+1\), then every diagram of size \(a_2\) is a sub-diagram of \(\lambda\), but the single-column diagram \((1^{a_2})\) is not a sub-diagram of \(\mu\) (since \(b_2'+1<a_2\)), which is a contradiction (see \Figure{=} for an illustration); similarly, if \(a_2'\leq b_2'+1\), then every diagram of size \(a_2'\) is a sub-diagram of \(\mu\), but the single-column diagram \((1^{a_2'})\) is not a sub-diagram of \(\lambda\) (since \(b_2+1<a_2'\)), which is a contradiction. Therefore, \[a_2-1 \geq b_2+1\text{ and }a_2'-1\geq b_2'+1.\] In view of \eqn{two=s}, they imply \(a_2-1=a_2'-1=b_2+1=b_2'+1\). Since \(|\lambda|=|\mu|\), we must then have \(\lambda=\mu\).
    \begin{figure}[h]
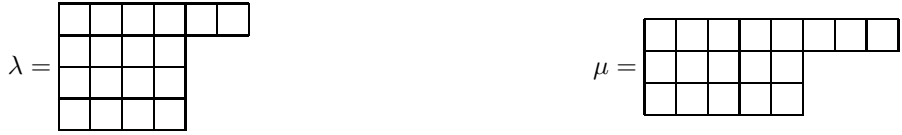

        \begin{subfigure}{0.49\textwidth}
            \centering
            \(\lambda = \tableau{12}{ [a]& [a]& [a]& [a]& [a]& [a]\\ [a]& [a]& [a]& [a]\\[a]& [a]& [a]& [a]\\[a]& [a]& [a]& [a]}\)
        \end{subfigure}
        \begin{subfigure}{0.49\textwidth}
            \centering
            \(\mu = \tableau{12}{ [a]& [a]& [a]& [a]& [a]& [a]& [a]& [a]\\ [a]& [a]& [a]& [a]& [a]\\[a]& [a]& [a]& [a]& [a]}\)
        \end{subfigure}
        \caption{The case \((\lambda^0)^T=\mu^0\) and \(a_2-1<b_2+1\).}
        \label{fig:=}
    \end{figure}

    Now assume \(r>2\).

    The subvarieties \(X_{\lambda^1},\ X_{\lambda^{r-1}}\) (respectively \(X'_{\mu^1},\ X'_{\mu^{r-1}}\)), are the only irreducible components of \(\Sing(X_\lambda)\) (respectively \(\Sing(X'_\mu)\)) that intersect all but one other irreducible component properly. Hence, \(X_{\lambda^1}\) and \(X_{\lambda^{r-1}}\) are isomorphic to \(X'_{\mu^1}\) and \(X'_{\mu^{r-1}}\) (not necessarily in this order). Using \Theorem{sing} we have the expressions
    \begin{equation*}
        \lambda^1 = ({(a_2-1)}^{b_2+1}, {a_3}^{b_3}, \dots, {a_r}^{b_r}),
    \end{equation*}
    \begin{equation*}
        \label{eqn:formulaofr-1}
        \lambda^{r-1} = (a_1, {a_2}^{b_2}, \dots, {a_{r-2}}^{b_{r-2}}, {a_{r-1}}^{b_{r-1}-1}, {(a_r-1)}^{b_r+1}),
    \end{equation*}
    and analogous expressions hold for \(\mu^1,\ \mu^{r-1}\). Note that
    \[
        \xi(\lambda^1)=a_2+b_2+\dots+b_r-1<a_1+b_2+\dots+b_r=\xi(\lambda^{r-1})
    \]
    and a similar inequality holds on the \(\mu\)-side. If \( X_{\lambda^1} \cong X'_{\mu^{r-1}}\), equivalently \( X_{\lambda^{r-1}} \cong X'_{\mu^1}\), then by the induction hypothesis we have that \(\lambda^1 = \mu^{r-1} \) or \( (\lambda^1)^T = \mu^{r-1} \) (the analogous equalities hold for \(\lambda^{r-1}, \mu^1\)). Hence,
    \[
        \xi(\lambda^{r-1}) > \xi(\lambda^1) = \xi(\mu^{r-1}) > \xi(\mu^1) = \xi(\lambda^{r-1}),
    \]
    a contradiction. Hence, we must have that \( X_{\lambda^1} \cong X'_{\mu^1}\) and \( X_{\lambda^{r-1}} \cong X'_{\mu^{r-1}}\). By the induction hypothesis we get that

    \begin{equation}
        \label{eqn:1}\text{either }\lambda^1=\mu^1\text{ or }(\lambda^1)^T=\mu^1;
    \end{equation}

    \begin{equation}
        \label{eqn:r-1}\text{either }\lambda^{r-1}=\mu^{r-1}\text{ or }(\lambda^{r-1})^T=\mu^{r-1}.
    \end{equation}

    If \(\lambda^1=\mu^1\), then \begin{equation*}a_2=a_2',\dots, \ a_r=a_r',\ b_2+b_3 = b_2'+b_3',\ b_4=b_4',\dots,\ b_r=b_r'.
    \end{equation*}
    Since \eqn{r-1} implies
    \(
    \xi(\lambda^{r-1})=\xi(\mu^{r-1}),
    \)
    we must also have \(a_1=a_1'\). Since \(|\lambda|=|\mu|\), we must then have \[b_2=b_2',\ b_3=b_3'.\] Therefore, \(\lambda=\mu\).

    If \(\lambda^{r-1}=\mu^{r-1}\), then
    \begin{equation*}
        a_1=a_1',\dots,\ a_{r-2}=a_{r-2}',\  a_r=a_r',\ b_2=b_2',\dots,\  b_r=b_r'.
    \end{equation*}
    Since \(|\lambda|=|\mu|\), we must then have \(a_{r-1} = a_{r-1}'\) and \(\lambda=\mu\).

    Now assume \((\lambda^1)^T=\mu^1\) and \((\lambda^{r-1})^T=\mu^{r-1}\). Note that
    \begin{equation*}
        % \label{eqn:formulaof1T}
        (\lambda^1)^T=((1+b_2+\dots+b_r)^{a_r}, (1+b_2+\dots+b_{r-1})^{a_{r-1}-a_r}, \dots,(1+b_2+b_3)^{a_3-a_4},(1+b_2)^{a_2-a_3-1})
    \end{equation*}
    and
    \begin{equation*}
        (\lambda^{r-1})^T=((1+b_2+\dots+b_r)^{a_r-1},(b_2+\dots+b_{r-1})^{a_{r-1}-a_r+1}, (1+b_2+\dots+b_{r-2})^{a_{r-2}-a_{r-1}},\dots,(1+b_2)^{a_2-a_3},1^{a_1-a_2}).
    \end{equation*}
    % (these identities hold even when \(a_2-1=a_3\) or \(b_{r-1}=1\)).
    The identity \((\lambda^{r-1})^T=\mu^{r-1}\) implies \(a_r=2\) and \((\lambda^1)^T=\mu^1\) implies \(a_r\geq b_2'+1\),  thus \(b_2'=1\).

    If \( r=3 \), then \((\lambda^{r-1})^T=\mu^{r-1}\) and the above imply that \[(1+b_2+b_3,b_2^{a_2-1},1^{a_1-a_2})=(\lambda^{r-1})^T=\mu^{r-1}=(a_1',(a_3'-1)^{b_3'+1}),\] which then
    implies \( a_2= 1 \), a contradiction to \(a_2 > a_3\).
    Otherwise, \(r > 3 \) and looking at the second rectangle from the top in \((\lambda^{r-1})^T=\mu^{r-1}\) gives
    \[
        a_{r-1}-a_r+1=b_2'\text{ when }b_{r-1}\neq1
    \]
    and \[
        (a_{r-1}-a_r+1)+(a_{r-2}-a_{r-1})=b_2'\text{ when }b_{r-1}=1.
    \]
    Either way, we must have \(b_2'>1\), which is again a contradiction to \(b_2'=1\) from above.
\end{proof}

\begin{lemma}\label{lemma:2}
    \Proposition{converse} holds when \(r_0+2=r\geq 2\).
\end{lemma}

\begin{proof}
    In this case
    \[
        \lambda=(a_1,{a_2}^{b_2},\dots,{a_{r-1}}^{b_{r-1}},1^{b_r}),\ \mu=(a_1',{a_2'}^{b_2'},\dots,{a_{r-1}'}^{b_{r-1}'},1^{b_r'}),
    \]
    with \( a_1 > a_2 > \dots >a_{r-1}>1,   a_1' > a_2' > \dots > a_{r-1}'>1\) and \( b_i,\ {b_i}' > 0 \) for \(i=2,\dots,r\).

    If \(r=2\), up to transposing some diagrams, we may assume \[a_1\leq b_2+1\text{ and }a_1'\leq b_2'+1.\] Suppose \(a_1>a_1'\). We count sub-diagrams of \(\lambda\) and \(\mu\) of size \( {a_1}'+1\). They are necessarily of the form \((c,1^d)\). \(\lambda \) contains \( a_1\geq {a_1}'+1 \) such sub-diagrams, while \(\mu\) contains fewer because the single-row diagram \((a_1'+1)\) is not a sub-diagram of \(\mu\) (see \Figure{rect2}). This is a contradiction. Therefore, \(a_1=a_1'\), \(b_2=b_2'\), i.e. \(\lambda=\mu\).
    \begin{figure}[h]
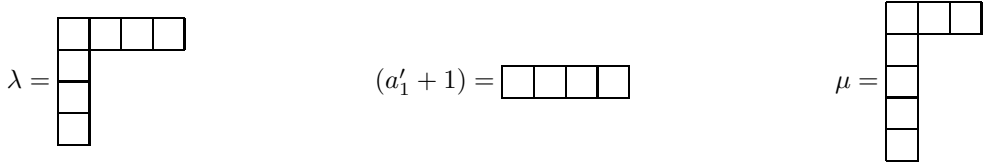

        \begin{subfigure}{0.32\textwidth}
            \centering
            \(\lambda=\tableau{12}{ [a] & [a] & [a]& [a]\\ [a]\\[a]\\[a]}\)
        \end{subfigure}
        \begin{subfigure}{0.32\textwidth}
            \centering
            \((a_1'+1) = \tableau{12}{ [a]& [a]& [a]& [a]}\)
        \end{subfigure}
        \begin{subfigure}{0.32\textwidth}
            \centering
            \(\mu=\tableau{12}{ [a] & [a]& [a]\\ [a]\\[a]\\[a]\\[a]}\)
        \end{subfigure}
        \caption{Diagrams \(\lambda,\ (a_1'+1),\ \mu\).}
        \label{fig:rect2}
    \end{figure}

    Now assume \(r>2\), and in this case \[\lambda^0=((a_2-1)^{b_2},\dots,(a_{r-1}-1)^{b_{r-1}}),\] and an analogous expression holds for \(\mu^0\). Again, we may assume \(\lambda^0=\mu^0\), which implies \begin{equation}\label{eqn:2tor-1}a_i=a_i',\ b_i=b_i'\text{ for }i=2,\dots,r-1.\end{equation}

    We have \[\lambda^1 = ({(a_2-1)}^{b_2+1}, {a_3}^{b_3}, \dots,{a_{r-1}}^{b_{r-1}}, 1^{b_r}),\] \[\lambda^{r-1}=(a_1, {a_2}^{b_2},\dots,{a_{r-2}}^{b_{r-2}},{a_{r-1}}^{b_{r-1}-1}),\] and similar expressions hold for \(\mu^1\) and \(\mu^{r-1}\). Considering proper intersections of components as in the proof of \Lemma{1}, we must have that either
    \[
        X_{\lambda^1} \cong X'_{\mu^1} \text{ and } X_{\lambda^{r-1}} \cong X'_{\mu^{r-1}}
    \]
    or
    \[
        X_{\lambda^1} \cong X'_{\mu^{r-1}} \text{ and } X_{\lambda^{r-1}} \cong X'_{\mu^1}.
    \]
    Since the first row of \(\lambda^1 \) has length \( a_2-1 \), but the first row of \( \mu^{r-1} \) has length \( a_1' > a_2' = a_2 \), we have \( \lambda^1 \neq \mu^{r-1}\). Similarly, \(\lambda^{r-1} \neq \mu^1 \). Since the left-most two columns of \(\lambda^1 \) are of distinct lengths, but the top-most two rows of \( {\mu^1} \) are of the same length, we have \( (\lambda^1)^T \neq \mu^1\). Similarly, \( (\lambda^{r-1})^T \neq \mu^{r-1} \). Therefore, we are in either of the following two cases:
    \begin{enumerate}
        \item[\namedlabel{itm:a}{(a)}] \(\lambda^1=\mu^1\) and \(\lambda^{r-1}=\mu^{r-1}\);
        \item[\namedlabel{itm:b}{(b)}] \((\lambda^1)^T=\mu^{r-1}\) and \((\lambda^{r-1})^T=\mu^1\).
    \end{enumerate}

    Assume \(r=3\). In case \ref{itm:a}, \(\lambda^1=\mu^1\) implies \[a_2=a_2',\ b_2+b_3=b_2'+b_3',\] and \(\lambda^2=\mu^2\) implies \[a_1=a_1',\ b_2=b_2'.\] Together, they imply \(\lambda=\mu\). In case \ref{itm:b}, \[\lambda^0=\mu^0 = \mu^1 \cap \mu^{2} = (\lambda^{2})^T \cap (\lambda^{1})^T = (\lambda^T)^1\cap(\lambda^T)^2 = (\lambda^T)^0.\] So we may transpose \(\lambda\) to reduce to case \ref{itm:a}.

    Finally, assume \(r>3\). In case \ref{itm:a}, \[\lambda = \lambda^1 \cup \lambda^{r-1} =\mu^1 \cup \mu^{r-1} = \mu;\] in case \ref{itm:b}, \[\lambda^T=  (\lambda^T)^1\cup(\lambda^T)^{r-1}= (\lambda^{r-1})^T\cup(\lambda^1)^T = \mu^1 \cup \mu^{r-1} = \mu.\]
\end{proof}

\begin{thm} \label{thm:main}
    Let \(X,\ X'\) be two Grassmannians and let \(X_\lambda\subseteq X,\ X'_\mu\subseteq X'\) be Schubert varieties. Then \(X_\lambda\cong X'_\mu\) as varieties if and only if \(\lambda=\mu\text{ or }\lambda^T = \mu\).
\end{thm}
\begin{proof}
    It follows from \Lemma{forward} and \Proposition{converse}.
\end{proof}

\section{The isomorphism problem for Richardson varieties}\label{sec:rich}

One can extend this problem to Richardson varieties, which are intersections of a Schubert variety \(X_{\mu^\vee}\) and an opposite Schubert variety \(X^\lambda\) in a Grassmannian. If the diagram \(\lambda\) is not contained in the diagram \(\mu^\vee\), then the Richardson variety is empty. We restrict our attention to non-empty Richardson varieties \(X_\theta\), which are indexed by skew diagrams \(\theta = \mu^\vee/\lambda\) consisting of boxes in \(\mu^\vee\) that are not in \(\lambda\), where \(\lambda\subseteq\mu^\vee\). While the definition of \(X_\theta\) depends on \(\lambda\) and \(\mu\), it follows from \cite[Lemma 3.2]{buch_pieri_2012} that the isomorphism class of \(X_\theta\) depends only on the skew diagram \(\theta\), where only the boxes in \(\theta\) and their relative positions are remembered. For two skew diagrams \(\theta\), \(\theta'\), we write \(\theta=\theta'\) if they are the same in this sense (see \Figure{relative}).

\begin{figure}[H]
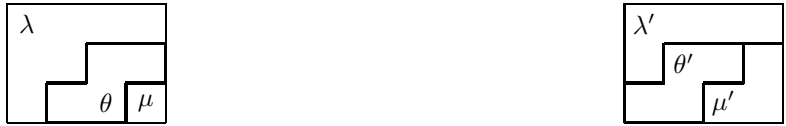

    \begin{subfigure}{0.49\textwidth}
        \[\tableau{15}{
            [tl] \lambda & [t]& [tB]& [tBr]\\
            [l] & [RB]  & []  & [BR]\\
            [lbR]& [B]& [BR] \theta & [rb]\mu }\]
    \end{subfigure}
    \begin{subfigure}{0.49\textwidth}
        \[\tableau{15}{
            [tl] \lambda' & [tB]& [tB]& [tbr]\\
            [lRB] & [] \theta' & [BR] & [tr]\\
            [LB]& [BR]& [b] \mu'& [br]}\]
    \end{subfigure}
    \caption{An example where \(\theta=\theta'\).}
    \label{fig:relative}
\end{figure}

Let \( \theta \) be a skew diagram, we define \(P(\theta) \) to be the poset on the boxes of \( \theta \) with the ordering induced by the covering relations \[ x < y \text{ if }x\text{ is immediately to the left or above }y.\] See \Figure{posetex} for an illustration. As before, we use \(|\theta| = |P(\theta)|\) to denote the number of boxes in \( \theta \).

For a poset \(P\), we will write \(P^{op}\) for its opposite poset, where the underlying set is the same but all orders are reversed. We adopt the definitions related to the connectedness of a poset in \cite{schroder2003ordered}. We say that two posets \(P,\ Q\) are \textit{semi-isomorphic}, denoted as \(P\sim Q\), if there is a bijection \(f\) between their sets of connected components, and for each connected component \(P_i\) of \(P\), either \(P_i\cong f(P_i)\) or \(P_i^{op}\cong f(P_i)\).

\begin{figure}[H]
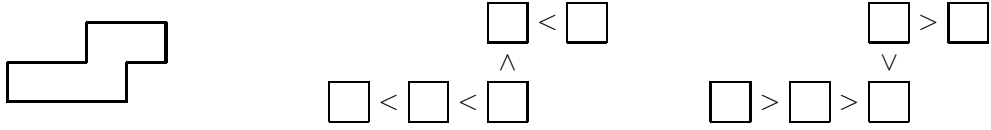

    \begin{subfigure}{0.30\textwidth}
        \[\tableau{15}{
            &[R]  & [T] & [TBR]\\
            [LBT]&[TB]& [BR] &  }\]
    \end{subfigure}
    \begin{subfigure}{0.30\textwidth}
        \[\tableau{15}{
            &&&     & [a] & []< & [a]\\
            && &     & [] \wedge & & \\
            [a]& [] < & [a]& []< & [a]  & &   }\]
    \end{subfigure}
    \begin{subfigure}{0.30\textwidth}
        \[\tableau{15}{
            &&&     & [a] & []> & [a]\\
            && &     & [] \vee & & \\
            [a]& [] > & [a]& []> & [a]  & &   }\]
    \end{subfigure}
    \caption{Example of \(\theta\), \(P(\theta)\), and \(P(\theta)^{op}\).}
    \label{fig:posetex}
\end{figure}

Let \(X,\ X'\) be two Grassmannians and \(X_\theta\subseteq X,\ X'_{\theta'}\subseteq X'\) be Richardson varieties. We will prove:
\begin{prop}\label{prop:if}
    If \(P(\theta)\sim P(\theta')\), then \(X_\theta\cong X'_{\theta'}\) as varieties.
\end{prop}

We conjecture the converse:
\begin{conj}\label{conj:only if}
    If \(X_\theta\cong X'_{\theta'}\) as varieties, then \( P(\theta) \sim P({\theta'}) \).
\end{conj}

While proving \Proposition{if}, we will show that if \(\lambda\) and \(\lambda'\) are diagrams, then \(P(\lambda)\cong P(\lambda')\) is equivalent to \[\text{either }\lambda=\lambda'\text{ or }\lambda^T=\lambda'\] (note that \(P(\lambda)^{op}\cong P(\lambda')\) is possible only when \(\lambda\) is a rectangle). Therefore, combining \Proposition{if} and \Conjecture{only if} would generalize \Theorem{main}.

We say that a skew diagram \(\theta\) is disconnected if it can be split into a non-empty NE part \(\theta'\) and a non-empty SW part \(\theta''\). More precisely, a NW-SE diagonal partitions the diagram into two disjoint non-empty skew diagrams. See \Figure{DisconnSkew}.
In the following, by diagonals we mean NW-SE diagonals. We call a diagonal ``higher" than another diagonal if it is strictly NE of the latter.
\begin{figure}[H]
    \[
        \begin{tikzpicture}[scale=0.4]
            \draw[-] (0, 1) -- (0,3) -- (1,3) -- (1,4) -- (3,4) -- (3,2) --(2,2) --(2,1) -- (0,1);
            \draw[-] (3,4) -- (3,6) -- (4,6) -- (4,7) -- (7,7) -- (7,6) --(6,6) --(6,4) -- (3,4);
            \draw[thick, -] (0,7) -- (6,1);
            \node[] at (5,5) {\(\theta'\)};
            \node[] at (1,2) {\(\theta''\)};
        \end{tikzpicture}
    \]
    \caption{A disconnected skew diagram.}
    \label{fig:DisconnSkew}
\end{figure}
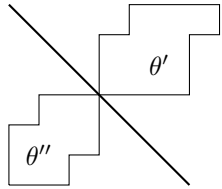

\begin{lemma}\label{lemma:conn}
    A skew diagram \(\theta\) is connected if and only if \(P(\theta)\) is connected.
\end{lemma}

\begin{proof}
    Assume \(\theta\) is connected.

    We claim that there exists a unique \( t \in \theta \) that is ``the most NE", i.e. that belongs to the highest diagonal intersecting \(\theta\).
    Otherwise, pick two such boxes \(t_1, t_2\in \theta\), since they must belong to the same diagonal, they are comparable. Assume that \(t_1<t_2\), using the fact that \(\theta\) is a skew diagram, there exists a box \(t'\in \theta\) that is east of \(t_1\) and north of \(t_2\) (see \Figure{DiagSkew}). Then \(t'\) belongs to a higher diagonal, a contradiction.

    \begin{figure}[H]
        \[
            \begin{tikzpicture}[scale=0.4]
                \node[draw] at (2,5) {\(t_1\)};
                \node[draw] at (5,2) {\(t_2\)};
                \node[draw] at (5,5) {\(t'\)};
                \draw[dashed] (1,6) -- (0,7);
                \draw[dashed] (3,4) -- (4,3);
                \draw[dashed] (7,0) -- (6,1);
                % \draw[-] (0, 1) -- (0,3) -- (1,3) -- (1,4) -- (3,4) -- (3,2) --(2,2) --(2,1) -- (0,1);
                % \draw[-] (3,4) -- (3,6) -- (4,6) -- (4,7) -- (7,7) -- (7,6) --(6,6) --(6,4) -- (3,4);
                % \draw[-, dashed] (3,6) -- (4,6) -- (4,5) -- (3,5);
                % \draw[-, dashed] (4,5) -- (5,5) -- (5,4) -- (4,4) -- (4,5);
                % \draw[-,blue, dashed] (4,6) -- (5,6) -- (5,5) -- (4,5) -- (4,6);
                % \draw[thick, -] (2,7) -- (6,3);
                % \draw[blue, -] (3,7) -- (7,3);
            \end{tikzpicture}
        \]
        \caption{Finding a box on a higher diagonal.}
        \label{fig:DiagSkew}
    \end{figure}
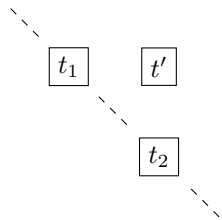

    If \( a \in \theta \setminus \{t\}\), then we claim that there exists an adjacent box on a higher diagonal. Assume the contrary, then since \(\theta\) is a skew diagram, there must be no box north of \(a\) (sharing the same column) and no box east of \(a\) (sharing the same row). Consequently, there must be no box strictly NW of \(a\) and no box strictly SE of \(a\) (without sharing the same row or column in both cases). It follows that the diagonal passing through \(a\)'s NE corner partitions \(\theta\) with \(a\) and \(t\) on different sides, contradicting the fact that \(\theta\) is connected.

    Repeatedly applying the above argument yields a path of adjacent boxes from \(a\) to \(t\) in \(P(\theta)\), hence \(P(\theta)\) is connected.

    Conversely, assume \(P(\theta)\) is connected. Let \(x\) and \(y\) be two different boxes in \(\theta\). Then there exists a sequence \((z_0 = x,\ z_1,\ \dots,\ z_l = y)\) of distinct boxes in \(\theta\) where consecutive boxes are adjacent. Since adjacent boxes in \(\theta\) cannot be separated by a diagonal, \(x,y\) are in the same connected component of \(\theta\). By our arbitrary choice of \(x,y\), we have that \( \theta\) is connected.
\end{proof}

\begin{defn}
    For \(\theta=\mu^\vee/\lambda\), let \(\theta^\dagger\coloneqq\lambda^\vee/\mu.\) See \Figure{Dagger}.
\end{defn}

\begin{figure}[H]
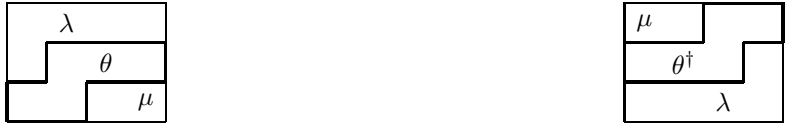

    \begin{subfigure}{0.49\textwidth}
        \[\tableau{15}{ [tl]& [tB]\lambda& [tB]& [tBr]\\ [lRB] & []  & [TB] \theta & [BR]\\[LH]& [BR]& [b] & [br] \mu }\]
    \end{subfigure}
    \begin{subfigure}{0.49\textwidth}
        \[\tableau{15}{ [tlB] \mu & [tBR] & [T]& [HR]\\ [LH] & [H] \theta^\dagger & [RB] & [Tr]\\[Tlb]& [Tb]& [Tb] \lambda & [br] }\]
    \end{subfigure}
    \caption{Skew diagrams \(\theta\) and \(\theta^\dagger\).}
    \label{fig:Dagger}
\end{figure}

Note that \(P(\theta^\dagger)=P(\theta)^{op}\).
Furthermore, considering the automorphism \(w_0\) of \(X\) we have that
\[
    X_{\theta^\dagger} = X_{\lambda^\vee} \cap X^\mu = w_0( X^\lambda) \cap w_0(X_{\mu^\vee}) = w_0(X^\lambda \cap X_{\mu^\vee}) = w_0(X_\theta).
\]
Thus, \(X_{\theta} \cong X_{\theta^\dagger}\).

Let \( \theta \) be a skew diagram and \( \theta^1, \dots, \theta^r \subseteq \theta \) be its connected components.
Then, \cite[Lemma 3.2(b)]{buch_pieri_2012} shows that
\[
    X_\theta \cong \prod X_{\theta^i}.
\] Therefore, \Proposition{if} and \Conjecture{only if} can be reduced to the case where \(\theta\) and \(\theta'\) are connected.

\begin{lemma}
    \label{lemma:strongskew}
    Let \( \theta,\ {\theta'} \) be connected skew diagrams and let \( \phi: P(\theta) \to P({\theta'}) \) be an isomorphism of posets, then either \[ \theta = {\theta'} \text{ and }\phi = \id \] or \[ \theta^T=\theta' \text{ and }\phi\text{ is given by transposing }\theta.\]
\end{lemma}
\begin{proof}
    Note that \( | \theta | = | P(\theta) | = | P({\theta'}) | = |{\theta'}| \).

    We will prove the statement by induction on \( | \theta | = |{\theta'}| \). If \( |\theta| = |{\theta'}| \leq 2 \), then the claim is trivially true, thus we assume that \( |\theta| > 2 \).

    Assume by induction that the lemma is true for all connected skew diagrams of size less than \( |\theta| \).

    Let \(t\) and \(b\) be the top and bottom boxes in the left-most non-empty column in \(\theta \), respectively. Suppose there is no box to the right of \(b\) in \(\theta\). Since \(\theta\) is connected and \(|\theta|>2\) by our assumption, there is a neighboring box \(y\) above \(b\). See \Figure{noboxrightofb}.

    Note that
    \[
        (\theta^\dagger)^T=(\lambda^\vee)^T/\mu^T=(\lambda^T)^\vee/\mu^T=(\theta^T)^\dagger\text{ and }\theta^{\dagger\dagger}=\theta.
    \]
    Therefore, we may replace \(\theta\) with \((\theta^T)^\dagger\), \(\theta'\) with \((\theta'^T)^\dagger\), and \(\phi\) with the corresponding isomorphism \[P((\theta^T)^\dagger)\cong P(\theta)^{op}\to P((\theta'^T)^\dagger)\cong P(\theta')^{op}\] if necessary and assume that there is a neighboring box \(y\) to the right of \(b\) in \(\theta\). See \Figure{dagger} and \Figure{lemma6case1}.
    \begin{figure}[H]
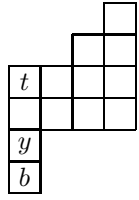
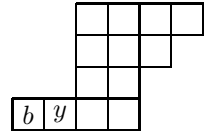

        \begin{subfigure}{0.49\textwidth}
            \centering
            \(\tableau{12}{ & & & [a]\\ & & [a]& [a]\\[a]t& [a]& [a]& [a]\\[a]& [a]& [a]& [a]\\[a]y\\[a]b}\)
            \caption{A skew diagram \(\theta\) where no box is to the right of \(b\).}
            \label{fig:noboxrightofb}
        \end{subfigure}
        \begin{subfigure}{0.49\textwidth}
            \centering
            \(\tableau{12}{ & & [a] & [a] & [a] & [a]\\ & & [a] & [a] & [a]\\ & & [a] & [a]\\[a]b & [a]y & [a] & [a]}\)
            \caption{The skew diagram \((\theta^T)^\dagger\).}
            \label{fig:dagger}
        \end{subfigure}
        \caption{Example of reducing to the case where there is a box to the right of $b$.}
    \end{figure}
    \begin{figure}[h]
        \begin{subfigure}{0.49\textwidth}
            \centering
            \(\tableau{12}{& & & [a]\\ & & [a]& [a]\\[a]t& [a]z& [a]& [a]\\[a]w & [a]& [a]& [a]\\[a]& [a]\\[a]b& [a]y}\)
        \end{subfigure}
        \begin{subfigure}{0.49\textwidth}
            \centering
            \(\tableau{12}{ & & [a] & [a] & [a] & [a]\\ & & [a] & [a] & [a]\\ & & [a] & [a]\\[a]t & [a]y & [a] & [a]}\)
        \end{subfigure}
        \caption{Example of skew diagrams where there is a box to the right of \(b\).}
        \label{fig:lemma6case1}
    \end{figure}

    Note that \(t\) is minimal in \(P(\theta)\), so that \(\theta_t:=\theta \setminus \{t\}\) is also a skew diagram and \(P(\theta_t) = P(\theta) \setminus \{t\}\).
    Since \(\theta\) is a skew diagram, there must then be a neighboring box \(z\) to the right of \(t\) and a column of boxes connecting \(y\) and \(z\) in \(\theta\). This implies that \(\theta_t\) is connected. In this case, \( \phi(t) \) is minimal in \(P(\theta')\), so \( {\theta'}_{\phi(t)} \) is also a skew diagram. Note that the restriction \( \phi_t: P(\theta_t) \to P({\theta'}_{\phi(t)}) \) is an isomorphism of posets. Using the fact that \( P(\theta_t) \cong P({\theta'}_{\phi(t)}) \) is connected, by \Lemma{conn} we have that \({\theta'}_{\phi(t)}\) is also connected.

    By induction hypothesis, up to transposing \(\theta'\), we may assume without loss of generality that \[\theta_t=\theta'_{\phi(t)}\text{ and }\phi_t=\id.\] It suffices to show that \(\phi(t)\) is the left neighbor of \(\phi(z)\) in \(\theta'\).

    If \(t \neq b\), then there is a neighboring box \(w\) below \(t\). \(\phi: P(\theta)\to P(\theta')\) being an isomorphism implies \(\phi(t)<\phi(z)\) and \(\phi(t)<\phi(w)\). There is a unique way to attach \(\phi(t)\) to \(\theta'_{\phi(t)}\) preserving these relations such that the result is a skew diagram.

    Otherwise, \(t=b\), and \(y=z\) is the only neighbor of \(t\) in \(\theta\). \(\phi: P(\theta)\to P(\theta')\) being an isomorphism implies that \(\phi(y)\) is the only neighbor of \(\phi(t)\) in \(\theta'\) and that \(\phi(t)<\phi(y)\). Suppose for a contradiction that \(\phi(t)\) is the top neighbor of \(\phi(y)\) in \(\theta'\). Then \(\phi_t=\id\) implies that there is no box above \(y\) in \(\theta_t\). Since \(\theta_t\) is a connected skew diagram and \(|\theta_t|\geq2\) by our assumption, there must be a neighboring box in \(\theta_t\) to the right of \(y\), and therefore a neighboring box to the right of \(\phi(y)\) in \(\theta'_{\phi(y)}\). Since \(\theta'\) is a skew diagram, there must then be a neighboring box to right of \(\phi(t)\) in \(\theta'\). See \Figure{lemma6impossible}. But then \(\phi(t)\) has two neighbors, which is a contradiction.
    \begin{figure}[h]
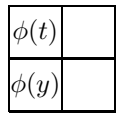

        \centering
        \(
        \tableau{20}{[a]\phi(t)& [a]\\ [a]\phi(y) & [a]}
        \)
        \caption{Impossible subset of \(\theta'\).}
        \label{fig:lemma6impossible}
    \end{figure}
\end{proof}

Since \(P(\theta^T) \cong P(\theta) \), we have the following:

\begin{cor}
    \label{cor:cor12}
    Let \( \theta,\ {\theta'} \) be connected skew diagrams, then \( P(\theta) \cong P({\theta'}) \) if and only if either \( \theta = {\theta'} \) or \( \theta^T=\theta' \).
\end{cor}

Finally, using the above corollary we obtain the following proof:

\begin{proof}[Proof of \Proposition{if}]
    Without loss of generality, assume \(\theta\) and \(\theta'\) are connected.

    Assume that \(P(\theta)\cong P(\theta')\), then by \Corollary{cor12} either \( \theta = {\theta'} \) or \( \theta^T=\theta' \). By \Lemma{transpose}, we may assume \( \theta = {\theta'} \). Using embeddings as in the proof of \Theorem{main} we can reduce to the case \(X = X'\) from which the conclusion follows.

    Otherwise, we have that \(P(\theta)^{op}\cong P(\theta')\). Consider the automorphism \(w_0\) of \(X\), it restricts to an isomorphism \(X_\theta \cong X_{\theta^\dagger}\). Since \(P(\theta^\dagger) \cong P(\theta)^{op} \cong P(\theta') \), by the previous case \(X_{\theta^\dagger} \cong X'_{\theta'}\) and the conclusion follows.
\end{proof}
% \input{references}

% %\bibliography{/home/asbuch/TeX/BibTeX/database,qhunique}
% \bibliographystyle{amsplain}
\bibliography{bib.bib}
\end{document}